\documentclass[11pt]{amsart}

\def\R{{\mathbb {R}}}
\def\N{{\mathbb {N}}}

\def\A{{\mathcal {A}}}

\newtheorem{teo}{Theorem}[section]
\newtheorem{lema}[teo]{Lemma}
\newtheorem{prop}[teo]{Proposition}

\theoremstyle{remark}
\newtheorem{remark}[teo]{Remark}

\theoremstyle{definition}

\numberwithin{equation}{section}

\begin{document}
\title[Sobolev embedding for variable exponent spaces]{On the Sobolev embedding theorem for variable exponent spaces in the critical range}

\author[J. Fern\'andez Bonder, N. Saintier and A. Silva]{Juli\'an Fern\'andez Bonder, Nicolas Saintier and Analia Silva}

\address[J. Fern\'andez Bonder and A. Silva]{IMAS - CONICET and Departamento de Matem\'atica, FCEyN - Universidad de Buenos Aires\hfill\break \indent Ciudad Universitaria, Pabell\'on I  (1428) Buenos Aires, Argentina.}

\address[N. Saintier]{UNGS and Departamento de Matem\'atica, FCEyN - Universidad de Buenos Aires\hfill\break \indent Ciudad Universitaria, Pabell\'on I  (1428) Buenos Aires, Argentina.}

\email[J. Fernandez Bonder]{jfbonder@dm.uba.ar}

\urladdr[J. Fernandez Bonder]{http://mate.dm.uba.ar/~jfbonder}

\email[A. Silva]{asilva@dm.uba.ar}

\email[N. Saintier]{nsaintie@dm.uba.ar}

\urladdr[N. Saintier]{http://mate.dm.uba.ar/~nsaintie}



\begin{abstract}
In this paper we study the Sobolev embedding theorem for variable exponent spaces with critical exponents. We find conditions on the best constant in order to guaranty the existence of extremals.
The proof is based on a suitable refinement of the estimates in the Concentration--Compactness Theorem for variable exponents and an adaptation of a convexity argument due to  P.L. Lions, F. Pacella and M. Tricarico.
\end{abstract}

\maketitle

\section{Introduction}
In this paper we study the existence problem for extremals of the Sobolev immersion Theorem for variable exponents $W^{1,p(x)}_0(\Omega)\hookrightarrow L^{q(x)}(\Omega)$. By extremals we mean functions where the following infimum is attained
\begin{equation}\label{Spq}
S(p(\cdot),q(\cdot),\Omega) = \inf_{v\in W^{1,p(x)}_0(\Omega)}\frac{\|\nabla v\|_{L^{p(x)}(\Omega)}}{\|v\|_{L^{q(x)}(\Omega)}}.
\end{equation}
Here $\Omega\subset\R^N$ is a bounded open set and the variable exponent spaces $L^{q(x)}(\Omega)$ and $W^{1,p(x)}_0(\Omega)$ are defined in the usual way. We refer to the book \cite{libro} for the definition and properties of these spaces, though in Section 2 we review the results relevant for this paper.

The {\em critical exponent} is defined as usual
$$
p^*(x)=\begin{cases}
\frac{Np(x)}{N-p(x)} & \mbox{if } p(x)<N,\\
\infty & \mbox{if } p(x)\ge N.
\end{cases}
$$

When the exponent $q(x)$ is {\em subcritical}, i.e. $1\le q(x)< p^*(x)-\delta$ for some $\delta>0$, the immersion is compact (see \cite{Fan}, Theorem 2.3), so the existence of extremals follows easily by direct minimization. But when the subcriticality is violated, i.e. $1\le q(x)\le p^*(x)$ with $\A = \{x\in\Omega\colon q(x)=p^*(x),\ p(x)<N\}\not=\emptyset$ the compactness of the immersion fails and so the existence (or not) of minimizers is not clear.
For instance, in the constant exponent case, it is well known that extremals do not exists for any bounded open set $\Omega$.

There are some cases where the subcriticality is violated but still the immersion is compact. In fact, in \cite{MOSS}, it is proved that if the criticality set is ``small'' and we have a control on how the exponent $q$ reaches $p^*$ at the criticality set, then the immersion $W^{1,p(x)}_0(\Omega)\hookrightarrow L^{q(x)}(\Omega)$ is compact, and so the existence of extremals follows as in the subcritical case.

However, in the general case $\A\not=\emptyset$, up to our knowledge, there are no results regarding the existence or not of extremals for the Sobolev immersion Theorem. This paper is an attempt to fill this gap.

In order to state our main results, let us introduce some notation.
\begin{itemize}
\item The Rayleigh quotient will be denoted by
\begin{equation}\label{Rayleigh}
 Q_{p,q,\Omega}(v):=\frac{\|\nabla v\|_{L^{p(x)}(\Omega)}}{\| v\|_{L^{q(x)}(\Omega)}}.
\end{equation}

\item The Sobolev immersion constant by
\begin{equation}\label{constante 2}
 S(p(\cdot),q(\cdot),\Omega) = \inf_{v\in W^{1,p(x)}_0(\Omega)}Q_{p,q,\Omega}(v).
\end{equation}

\item The localized Sobolev constant by
\begin{equation}\label{constante 4}
\overline{S_x}=\sup_{\varepsilon>0} S(p(\cdot),q(\cdot),B_\varepsilon(x)) = \lim_{\varepsilon\to 0+}S(p(\cdot),q(\cdot),B_\varepsilon(x)),\qquad x\in\Omega.
\end{equation}

\item The critical constant by
\begin{equation}
\overline{S}=\inf_{x\in\A}\overline{S_x}.
\end{equation}

\item The usual Sobolev constant for constant exponents
\begin{equation}\label{constante 3}
K^{-1}_{r}=\inf_{v\in C^{\infty}_c(\R^n)}\frac{\|\nabla v\|_{L^{r}(\R^N)}}{\|v\|_{L^{r^*}(\R^N)}}.
\end{equation}
\end{itemize}

With these notations, our main results can be stated as
\begin{teo}\label{estimacion}
Assume that $p(\cdot), q(\cdot)\colon \Omega\to \R$ are continuous functions with modulus of continuity $\rho(t)$ such that
$$
\rho(t)\log(1/t)\to 0\quad \mbox{as } t\to 0+.
$$
Assume, moreover, that the criticality set $\A$ is nonempty and $\sup_{\Omega}p(\cdot)\le \inf_{\Omega}q(\cdot)$.

Then, for every domain $\Omega$ it holds
$$
S(p(\cdot),q(\cdot),\Omega) \le \overline S \le \inf_{p^-_{\A}\le r\le p^+_{\A}} K^{-1}_r,
$$
where $p^-_{\A}:=\inf_{\A}p(\cdot)$ and $p^+_{\A}:=\sup_{\A}p(\cdot)$.
\end{teo}

\begin{teo}\label{existencia}
Under the same assumptions of the previous Theorem, if  the strict inequality holds
$$
S(p(\cdot),q(\cdot),\Omega) < \overline S,
$$
then there exists an extremal for the immersion $W^{1,p(x)}_0(\Omega)\hookrightarrow L^{q(x)}(\Omega)$.
\end{teo}

These two theorems give rise to two natural questions:
\begin{enumerate}
\item Is $\overline S = \inf_{p^-_{\A}\le r \le p^+_{\A}} K^{-1}_r$ or the inequality is strict?

\item For what domains $\Omega$ and exponents $p(x)$, $q(x)$ is the strict inequality $S(p(\cdot),q(\cdot),\Omega) < \overline S$ achieved?
\end{enumerate}

We give partial answer to these questions in this paper. For question (1) we show that
$\overline{S_x} = K^{-1}_{p(x)}$ for every point $x\in\A$ which is a local minimum of $p$ and a local maximum of $q$. As far as we know, it is an open problem to determine wether this inequality holds in general or not. For question (2), we show that the strict inequality is achieved for every domain $\Omega$ such that the subcriticality set $\Omega\setminus\A$ contains a sufficiently large ball. It will be interesting to know if there exists an example of the strict inequality in the case where $q(x)=p^*(x)$ in $\Omega$.

In the course of our study of question (1), we need to show that the constant $S(p(\cdot), q(\cdot), \Omega)$ is continuous with respect to $p(\cdot)$ and $q(\cdot)$ in the $L^\infty(\Omega)$ topology for monotone sequences. We believe that this result has independent interest.

\medskip

The proof of Theorem \ref{existencia} heavily relies on the Concentration--Compactness Theorem for variable exponents that was proved independently by \cite{FBS1} and \cite{Fu}. Moreover, what is needed here is a slight refinement of the version in \cite{FBS1}. Though this refinement follows as a simple observation in \cite{FBS1}, we make here a sketch of the full proof of the Concentration--Compactness Theorem in order to make the paper self contained.

The other key ingredient in the proof is the adaptation of a convexity argument due to P.L. Lions, F. Pacella and M. Tricarico \cite{LPT} in order to show that a minimizing sequence either concentrates at a single point or is strongly convergent.

Analogous results can be obtained for the trace embedding theorem by applying similar techniques. See \cite{FBSS2}.

\medskip

To end this introduction, let us comment on different applications where the $p(x)-$Laplacian has appeared.

Up to our knowledge there are two main fields where the $p(x)-$Laplacian has been proved to be extremely useful in applications:
\begin{itemize}
\item Image Processing
\item Electrorheological Fluids
\end{itemize}

For instance, in \cite{CLR}, Y. Chen, S. Levin and R. Rao proposed the following model in image processing
$$
E(u)=\int_{\Omega}\frac{|\nabla u(x)|^{p(x)}}{p(x)}+ f(|u(x)-I(x)|)\, dx \to \mbox{min}
$$
where $p(x)$ is a function varying between $1$ and $2$ and $f$ is a convex function. In their application, they chose $p(x)$ close to 1 where there is likely to be edges and close to 2 where it is unlikely to be edges.

The electrorheological fluids application is much more developed and we refer to the monograph by M. Ru{\v{z}}i{\v{c}}ka, \cite{Ru}, and its references. In these models, after some simplifications, it leads to solve
\begin{equation}\label{ecuacion}
\begin{cases}
-\Delta_{p(x)} u = f(x,u,\nabla u) & \mbox{in }\Omega\\
u = 0 & \mbox{on }\partial\Omega
\end{cases}
\end{equation}
for some nonlinear source $f$. In most cases, the source term is taken to be only dependent on $u$ and so in order for the usual variational techniques to work, one needs a control on the growth of $f$ given by the Sobolev embedding. In this regard there are plenty of literature that deal with this problem (just to cite a few, see \cite{Cabada-Pouso, Dinu, FZ, Mihailescu, Mihailescu-Radulescu}).
When the source term has critical growth in the sense of the Sobolev embedding, there are only a few results on the existence of solutions for \eqref{ecuacion}. We refer to the above mentioned works of \cite{FBS1, Fu, MOSS} and also the work \cite{Silva} where multiplicity results for \eqref{ecuacion} are obtained.

\medskip

\subsection*{Organization of the paper}
The rest of the paper is organized as follows. In Section 2, we collect some preliminaries on variable exponent spaces that will be used throughout the paper. In Section 3 we revisit the proof of the Concentration--Compactness Theorem in the version of \cite{FBS1} in order to make the necessary refinement. In Section 4 we prove our main results, Theorem \ref{estimacion} and Theorem \ref{existencia}. In Section 5 we prove the continuity of the Sobolev constant with respect to $p$ and $q$ in the $L^\infty$ topology. In Section 6 we give partial answer to question (1) and show that for $x$ a local minimum of $p$ and local maximum of $q$, $\overline{S_x} = K^{-1}_{p(x)}$. Finally,  in Section 7 we give partial answer to question (2) and show that if $\Omega\setminus\A$ contains a sufficiently large ball, then $S(p(\cdot), q(\cdot),\Omega) < \overline{S}$.

\section{Preliminaries on variable exponent Sobolev spaces}

In this section we review some preliminary results regarding Lebesgue and Sobolev spaces with variable exponent. All of these results and a comprehensive study of these spaces can be found in \cite{libro}.

\medskip

The variable exponent Lebesgue space $L^{p(x)}(\Omega)$ is defined by
$$
L^{p(x)}(\Omega) = \Big\{u\in L^1_{\text{loc}}(\Omega) \colon \int_\Omega|u(x)|^{p(x)}\,dx<\infty\Big\}.
$$
This space is endowed with the norm
$$
\|u\|_{L^{p(x)}(\Omega)}=\inf\Big\{\lambda>0:\int_\Omega\Big|\frac{u(x)}{\lambda}\Big|^{p(x)}\,dx\leq 1\Big\}
$$
The variable exponent Sobolev space $W^{1,p(x)}(\Omega)$ is defined by
$$
W^{1,p(x)}(\Omega) = \{u\in W^{1,1}_{\text{loc}}(\Omega) \colon u\in L^{p(x)}(\Omega) \mbox{ and } |\nabla u |\in  L^{p(x)}(\Omega)\}.
$$
The corresponding norm for this space is
$$
\|u\|_{W^{1,p(x)}(\Omega)}=\|u\|_{L^{p(x)}(\Omega)}+\| \nabla u \|_{L^{p(x)}(\Omega)}
$$
Define $W^{1,p(x)}_0(\Omega)$ as the closure of $C_c^\infty(\Omega)$ with respect to the $W^{1,p(x)}(\Omega)$ norm. The spaces $L^{p(x)}(\Omega)$, $W^{1,p(x)}(\Omega)$ and $W^{1,p(x)}_0(\Omega)$ are separable and reflexive Banach spaces when $1<\inf_\Omega p \le \sup_\Omega p <\infty$.

As usual, we denote the conjugate exponent of $p(x)$ by $p'(x) = p(x)/(p(x)-1)$  and the Sobolev exponent by
$$
p^*(x)=\begin{cases}
\frac{Np(x)}{N-p(x)} & \mbox{ if } p(x)<N\\
\infty & \mbox{ if } p(x)\geq N
\end{cases}
$$

The following result is proved in \cite{Fan} (see also \cite{libro}, pp. 79, Lemma 3.2.20 (3.2.23)).

\begin{prop}[H\"older-type inequality]\label{Holder}
Let $f\in L^{p(x)}(\Omega)$ and $g\in L^{q(x)}(\Omega)$. Then the following inequality holds
$$
\| f(x)g(x) \|_{L^{s(x)}(\Omega)}\le  \Big( \Big(\frac{s}{p}\Big)^+ + \Big(\frac{s}{q}\Big)^+\Big) \|f\|_{L^{p(x)}(\Omega)}\|g\|_{L^{q(x)}(\Omega)},
$$
where
$$
\frac{1}{s(x)} = \frac{1}{p(x)} + \frac{1}{q(x)}.
$$
\end{prop}

The Sobolev embedding Theorem is also proved in \cite{Fan}, Theorem 2.3.

\begin{prop}[Sobolev embedding]\label{embedding}
Let $p, q\in C(\overline{\Omega})$ be such that $1\leq q(x)\le p^*(x)$ for all $x\in\overline{\Omega}$. Then there is a continuous embedding
$$
W^{1,p(x)}(\Omega)\hookrightarrow L^{q(x)}(\Omega).
$$
Moreover, if $\inf_{\Omega} (p^*-q)>0$ then, the embedding is compact.
\end{prop}

As in the constant exponent spaces, Poincar\'e inequality holds true (see \cite{libro}, pp. 249, Theorem 8.2.4)
\begin{prop}[Poincar\'e inequality]\label{Poincare}
Assume $p(x)$ is log-H\"older continuous. Then, there is a constant $C>0$, $C=C(\Omega)$, such that
$$
\|u\|_{L^{p(x)}(\Omega)}\leq C\|\nabla u\|_{W^{1,p(x)}(\Omega)},
$$
for all $u\in W^{1,p(x)}_0(\Omega)$.
\end{prop}

\begin{remark}
By Proposition \ref{Poincare}, we know that $\| \nabla u \|_{L^{p(x)}(\Omega)}$ and $\|u\|_{W^{1,p(x)}(\Omega)}$ are equivalent norms on $W_0^{1,p(x)}(\Omega)$.
\end{remark}

Throughout this paper the following notation will be used: Given $q\colon \Omega\to\R$ bounded, we denote
$$
q^+ := \sup_\Omega q(x), \qquad q^- := \inf_\Omega q(x).
$$

The following proposition is also proved in \cite{Fan} and it will be most usefull (see also \cite{libro}, Chapter 2, Section 1).
\begin{prop}\label{norma.y.rho}
Set $\rho(u):=\int_\Omega|u(x)|^{p(x)}\,dx$. For $u\in L^{p(x)}(\Omega)$ and $\{u_k\}_{k\in\N}\subset L^{p(x)}(\Omega)$, we have
\begin{align}
& u\neq 0 \Rightarrow \Big(\|u\|_{L^{p(x)}(\Omega)} = \lambda \Leftrightarrow \rho(\frac{u}{\lambda})=1\Big).\\
& \|u\|_{L^{p(x)}(\Omega)}<1 (=1; >1) \Leftrightarrow \rho(u)<1(=1;>1).\\
& \|u\|_{L^{p(x)}(\Omega)}>1 \Rightarrow \|u\|^{p^-}_{L^{p(x)}(\Omega)} \leq \rho(u) \leq \|u\|^{p^+}_{L^{p(x)}(\Omega)}.\\
& \|u\|_{L^{p(x)}(\Omega)}<1 \Rightarrow \|u\|^{p^+}_{L^{p(x)}(\Omega)} \leq \rho(u) \leq \|u\|^{p^-}_{L^{p(x)}(\Omega)}.\\
& \lim_{k\to\infty}\|u_k\|_{L^{p(x)}(\Omega)} = 0 \Leftrightarrow \lim_{k\to\infty}\rho(u_k)=0.\\
& \lim_{k\to\infty}\|u_k\|_{L^{p(x)}(\Omega)} = \infty \Leftrightarrow \lim_{k\to\infty}\rho(u_k) = \infty.
\end{align}
\end{prop}

For much more on these spaces, we refer to \cite{libro}.

\section{Refinement of the Concentration--Compactness Theorem}

In this section we make a refinement of the Concentration--Compactness Theorem for variable exponent spaces that was proved independently by \cite{FBS1} and \cite{Fu}.

The refinement made here is essential in the remaining of the paper and it involves a precise computation of the constants. More precisely, we prove

\begin{teo}\label{CCT}
Let $\{u_n\}_{n\in\N}\subset W^{1,p(x)}_0(\Omega)$ be a sequence such that $u_n\rightharpoonup u$ weakly in $W^{1,p(x)}_0(\Omega)$. Then there exists a finite set $I$ and points $\{x_i\}_{i\in I}\subset \A$ such that
\begin{align}
\label{nu}& |u_n|^{q(x)} \rightharpoonup \nu = |u|^{q(x)} + \sum_{i\in I} \nu_i \delta_{x_i} \qquad \text{weakly in the sense of measures}\\
\label{mu}& |\nabla u_n|^{p(x)} \rightharpoonup \mu \geq |\nabla u|^{p(x)} + \sum_{i\in I} \mu_i \delta_{x_i} \qquad \text{weakly in the sense of measures}\\
\label{refinement}& \overline{S}_{x_i}\nu_i^{\frac{1}{q(x_i)}}\leq \mu_i^{\frac{1}{p(x_i)}}.
\end{align}
\end{teo}

\begin{remark}
The refinement that we present here is in inequality \eqref{refinement}.
\end{remark}

\begin{proof}

As in \cite{FBS1} it is enough to consider the case where $u_n\rightharpoonup 0$ weakly in $W^{1,p(x)}_0(\Omega)$.

\medskip

Now, consider $\phi\in C^\infty_c(\Omega)$, from Sobolev inequality for variable exponents, we obtain
\begin{equation}\label{poincare}
S(p(\cdot),q(\cdot),\Omega) \|\phi u_n\|_{L^{q(x)}(\Omega)} \leq \|\nabla(\phi u_n)\|_{L^{p(x)}(\Omega)}.
\end{equation}
In order to compute the right-hand side, we observe that
\begin{equation}\label{former}
| \|\nabla(\phi u_n)\|_{L^{p(x)}(\Omega)}-\|\phi\nabla u_n\|_{L^{p(x)}(\Omega)}|\le \| u_n \nabla \phi\|_{L^{p(x)}(\Omega)}.
\end{equation}
Then, we see that the right hand side of \eqref{former} converges to 0. In fact,
$$
\| u_n \nabla \phi\|_{L^{p(x)}(\Omega)}\le (\|\nabla \phi\|_{L^{\infty}(\Omega)}+1)^{p^+}
\| u_n \|_{L^{p(x)}(\Omega)} \to 0,
$$
as $W^{1,p(x)}_0(\Omega)$ is compactly embedded in  $L^{p(x)}(\Omega)$.

Therefore, taking $n\to \infty$ in \eqref{poincare}, we have,
\begin{equation}\label{RH}
S(p(\cdot),q(\cdot),\Omega) \|\phi\|_{L_\nu^{q(x)}(\Omega)} \leq\|\phi\|_{L_\mu^{p(x)}(\Omega)}.
\end{equation}
This is a reverse-H\"older type inequality for the measures $\mu$ and $\nu$. Now, as in \cite{FBS1} it follows that \eqref{nu} and \eqref{mu} hold.

\medskip

Again, exactly as in \cite{FBS1} it follows that the points $\{x_i\}_{i\in I}$ belong to the {\em critical set} $\A$.

\medskip

It remains to see \eqref{refinement}. 

Let $\phi\in C^\infty_c(\R^N)$ be such that $0\leq\phi\leq1$, $\phi(0)=1$ and supp$(\phi)\subset B_1(0)$. Now, for each $i\in I$ and $\varepsilon>0$, we denote $\phi_{\varepsilon,i}(x):= \phi((x-x_i)/\varepsilon)$.

Since supp$(\phi_{\varepsilon,i} u_n)\subset B_{\varepsilon}(x_i)$, by \eqref{RH} with $\Omega=B_\varepsilon(x_i)$, we obtain
$$
S(p(\cdot),q(\cdot),B_\varepsilon(x_i)) \|\phi_{\varepsilon,i}\|_{L^{q(x)}_\nu(B_\varepsilon(x_i))} \leq \|\phi_{\varepsilon,i}\|_{L^{p(x)}_\mu(B_\varepsilon(x_i))}.
$$

By \eqref{nu}, we have
\begin{align*}
\rho_\nu(\phi_{i_0,\varepsilon}) &:= \int_{B_\varepsilon(x_{i_0})}|\phi_{i_0,\varepsilon}|^{q(x)}\,d\nu \\
&= \int_{B_\varepsilon(x_{i_0})} |\phi_{i_0,\varepsilon}|^{q(x)}|u|^{q(x)}\, dx + \sum_{i\in I} \nu_i\phi_{i_0,\varepsilon}(x_i)^{q(x_i)}\\
&\geq \nu_{i_0}.
\end{align*}

From now on, we will denote
\begin{align*}
q^+_{i,\varepsilon}:=\sup_{B_\varepsilon(x_i)}q(x),\qquad q^-_{i,\varepsilon}:=\inf_{B_\varepsilon(x_i)}q(x),\\
p^+_{i,\varepsilon}:=\sup_{B_\varepsilon(x_i)}p(x),\qquad p^-_{i,\varepsilon}:=\inf_{B_\varepsilon(x_i)}p(x).
\end{align*}

If $\rho_\nu(\phi_{i_0,\varepsilon})<1$ then
$$
 \|\phi_{i_0,\varepsilon}\|_{L^{q(x)}_\nu (B_\varepsilon(x_{i_0}))} \ge \rho_\nu(\phi_{i_0,\varepsilon})^{1/q^-_{i,\varepsilon}} \ge \nu_{i_0}^{1/q^-_{i,\varepsilon}}.
$$
Analogously, if $\rho_\nu(\phi_{i_0,\varepsilon})>1$ then
$$
\|\phi_{i_0,\varepsilon}\|_{L^{q(x)}_\nu(B_\varepsilon(x_{i_0}))} \ge \nu_{i_0}^{1/q^+_{i,\varepsilon}}.
$$
Therefore,
$$
\min\Big\{\nu_i^\frac{1}{q^+_{i,\varepsilon}}, \nu_i^\frac{1}{q^-_{i,\varepsilon}}\Big\} S(p(\cdot),q(\cdot),B_\varepsilon(x_i)) \leq \|\phi_{i,\varepsilon}\|_{L^{p(x)}_\mu(B_\varepsilon(x_i))} .
$$

On the other hand,
$$
\int_{B_\varepsilon(x_i)} |\phi_{i,\varepsilon}|^{p(x)}\,d\mu\leq \mu(B_{\varepsilon}(x_i))
$$
hence
\begin{align*}
 \|\phi_{i,\varepsilon}\|_{L^{p(x)}(B_\varepsilon(x_i))} &\leq  \max \Big\{ \rho_\mu (\phi_{i,\varepsilon})^\frac{1}{p^+_{i,\varepsilon}}, \rho_\mu(\phi_{i,\varepsilon})^\frac{1} {p^-_{i,\varepsilon}}\Big\}\\
&\le \max \Big\{ \mu(B_\varepsilon(x_i))^\frac{1}{p^+_{i,\varepsilon}}, \mu(B_\varepsilon(x_i))^\frac{1}{p^-_{i,\varepsilon}}\Big\},
\end{align*}
so we obtain,
$$
S(p(\cdot),q(\cdot),B_\varepsilon(x_i))\min \Big\{ \nu_i^\frac{1}{q^+_{i,\varepsilon}}, \nu_i^\frac{1}{q^-_{i,\varepsilon}}\Big\} \leq \max \Big\{ \mu(B_\varepsilon(x_i))^\frac{1}{p^+_{i,\varepsilon}}, \mu(B_\varepsilon(x_i))^\frac{1}{p^-_{i,\varepsilon}}\Big\}.
$$
As $p$ and $q$ are continuous functions and as $q(x_i) = p^*(x_i)$, letting $\varepsilon\to 0$, we get
$$
\Big(\lim_{\varepsilon\to 0} S(p(\cdot),q(\cdot),B_\varepsilon(x_i))\Big) \nu_i^{1/p^*(x_i)} \le \mu_i^{1/p(x_i)},
$$
where $\mu_i := \lim_{\varepsilon\to 0}\mu(B_\varepsilon(x_i))$.

The proof is now complete.
\end{proof}

\section{Proof of the main results}

We begin this section with the proof of Theorem \ref{estimacion}.

\subsection{Proof of Theorem \ref{estimacion}}

First we prove a uniform upper bound for $S(p(\cdot),q(\cdot),\Omega)$ depending only on $p^-_\A$ and $p^+_\A$.

\begin{lema}\label{desigualdad.primera}
With the assumptions of Theorem \ref{estimacion}, it holds that
$$
S(p(\cdot),q(\cdot),\Omega)\le \inf_{p^-_{\A}\le r\le p^+_{\A}} K^{-1}_r,
$$
where $K_r^{-1}$ is given in \eqref{constante 3}.
\end{lema}
\begin{proof}
First, we observe that our regularity assumptions on $p$ and $q$ implies that
\begin{align*}
& q(x_0+\lambda x) = q(x_0) + \rho_1(\lambda, x) = p^*(x_0) + \rho_1(\lambda, x),\\
& p(x_0+\lambda x) = p(x_0) + \rho_2(\lambda, x),
\end{align*}
with $\lim_{\lambda\to 0+}\lambda^{\rho_k(\lambda, x)} = 1$ uniformly in $\Omega$ ($k=1,2$).

Now, let $\phi\in C^\infty_c(\Omega)$, and define $\phi_\lambda$ to be the rescaled function around $x_0\in \A$ as $\phi_\lambda=\lambda^\frac{-n}{p^*(x_0)}\phi(\frac{x-x_0}{\lambda})$. Then we have
$$
1 = \int_\Omega \Big(\frac{\phi_\lambda}{\|\phi_\lambda\|_{L^{q(x)}(\Omega)}}\Big)^{q(x)}\,dx = \int_\Omega \lambda^{\frac{-N(p^*(x_0)+\rho_1(\lambda,x_0+\lambda y))}{p^*(x_0)}+N} \Big(\frac{\phi(y)}{{\|\phi_\lambda\|_{L^{q(x)}(\Omega)}}}\Big)^{q(x_0)+\rho_1(\lambda,x_0+\lambda y)}\,dy.
$$
Since
$$
\lambda^{\frac{-N \rho_1(\lambda,x_0+\lambda y)}{p^*(x_0)}}\Big(\frac{\phi(y)}{{\|\phi_\lambda\|_{L^{q(x)}(\Omega)}}}\Big)^{\rho_1(\lambda,x_0+\lambda y)}\to 1 \mbox{ when }\lambda\to 0+ \mbox{ in } \{|\phi|>0\}\subset  \Omega,
$$
we get
$$
1=\frac{\int_\Omega|\phi(y)|^{q(x_0)}\,dy}{\lim_{\lambda\to 0}\|\phi_\lambda\|_{L^{q(x)}(\Omega)}^{q(x_0)}}.
$$
Analogously,
\begin{align*}
1 &= \int_\Omega \Big(\frac{|\nabla\phi_\lambda|}{\|\nabla\phi_\lambda\|_{L^{p(x)}(\Omega)}}\Big)^{p(x)}\,dx \\
&= \int_\Omega \lambda^{\frac{-N(p(x_0)+\rho_2(\lambda,x_0+\lambda y))}{p^*(x_0)} + N} \Big(\frac{\frac{1}{\lambda} |\nabla\phi(y)|}{\|\nabla\phi_\lambda(y)\|_{L^{p(x)}(\Omega)}}\Big)^{p(x_0) + \rho_2(\lambda,x_0+\lambda y)}\,dx\\
&=\int_\Omega \lambda^{\frac{-N(p(x_0)+\rho_2(\lambda,x_0+\lambda y))}{p^*(x_0)}+N-p(x_0)-\rho_2(\lambda,x_0+\lambda y)} \Big(\frac{|\nabla\phi(y)|}{\|\nabla\phi_\lambda(y)\|_{L^{p(x)}(\Omega)}}\Big)^{p(x_0)+\rho_2(\lambda,x_0+\lambda y)}\,dx.
\end{align*}
Again,
$$
\lambda^{\frac{-N\rho_2(\lambda,x_0+\lambda y)}{p^*(x_0)}-\rho_2(\lambda,x_0+\lambda y)}\Big(\frac{|\nabla\phi(y)|}{\|\nabla\phi_\lambda(y)\|_{L^{p(x)}(\Omega)}}\Big)^{\rho_2(\lambda,x_0+\lambda y)}\to 1\mbox{ when }\lambda\to 0+ \mbox{ in } \{|\nabla \phi|>0\}\subset \Omega,
$$
so we arrive at
$$
1=\frac{\int_\Omega|\nabla\phi(y)|^{p(x_0)}\,dy}{\lim_{\lambda\to 0+}\|\nabla\phi_\lambda\|_{L^{p(x)}(\Omega)}^{p(x_0)}}.
$$

\medskip

Now, by definition of $S(p(\cdot),q(\cdot),\Omega)$,
$$
S(p(\cdot),q(\cdot),\Omega)\leq\frac{\|\nabla\phi_\lambda\|_{L^{p(x)}(\Omega)}}{\|\phi_\lambda\|_{L^{q(x)}(\Omega)}}
$$
and taking limit $\lambda\to 0+$, we obtain
$$
S(p(\cdot),q(\cdot),\Omega)\leq\frac{\|\nabla\phi\|_{L^{p(x_0)}(\Omega)}}{\|\phi\|_{L^{q(x_0)}(\Omega)}}
$$
for every $\phi\in C^\infty_c(\Omega)$. Then,
$$
S(p(\cdot),q(\cdot),\Omega)\leq K^{-1}_{p(x_0)},
$$
so,
$$
S(p(\cdot),q(\cdot),\Omega)\leq \inf_{p^-_\A\le r\le p^+_\A} K^{-1}_r
$$
as we wanted to show.
\end{proof}

Now, the proof of Theorem \ref{estimacion} follows easily as a simple corollary of Lemma \ref{desigualdad.primera}.

\begin{proof}[Proof of Theorem \ref{estimacion}]
Applying Lemma \ref{desigualdad.primera} to the case $\Omega=B_\varepsilon(x_0)$ for $x_0\in\A$ we get that
$$
S(p(\cdot),q(\cdot),B_\varepsilon(x_0))\le K^{-1}_{p(x_0)}
$$
for every $\varepsilon>0$. So
$$
\bar S_{x_0}\le K^{-1}_{p(x_0)}.
$$
Now, for the first inequality, we just observe that the Sobolev constant is nondecreasing with respect to inclusion, so
$$
S(p(\cdot),q(\cdot),\Omega)\le S(p(\cdot),q(\cdot),B_\varepsilon(x_0))
$$
for every ball $B_\varepsilon(x_0)\subset\Omega$.

So the result follows.
\end{proof}

\subsection{Proof of Theorem \ref{existencia}}

Now we focus on our second theorem. We begin by adapting a convexity argument used in \cite{LPT} to the variable exponent case.

\begin{teo}
Assume that $p^+< q^-$. Let $\{u_n\}_{n\in N}$ be a minimizing sequence for \eqref{constante 2}. Then the following alternative holds
\begin{itemize}
\item $\{u_n\}_{n\in \N}$ has a strongly convergence subsequence in $L^{q(x)}(\Omega)$ or

\item $\{u_n\}_{n\in\N}$ has a subsequence such that $|u_n|^{q(x)}\rightharpoonup \delta_{x_0}$ weakly in the sense of measures and $|\nabla u_n|^{p(x)}\rightharpoonup \overline{S}_{x_0}^{ p(x_0)}\delta_{x_0}$ weakly in the sense of measures, for some $x_0\in \A$.
\end{itemize}
\end{teo}

\begin{proof}
Let $\{u_n\}_{n\in\N}$ be a normalized minimizing sequence, that is,
$$
S(p(\cdot),q(\cdot),\Omega) = \lim_{n\to\infty}\|\nabla u_n\|_{L^{p(x)}(\Omega)}
$$
and
$$
\|u_n\|_{L^{q(x)}(\Omega)}=1
$$
Since $\{u_n\}_{n\in\N}$ is bounded in $W^{1,p(x)}_0(\Omega)$, by the Concentration--Compactness Theorem (Theorem \ref{CCT}), we have that, for a subsequence that we still denote by $\{u_n\}_{n\in\N}$,
\begin{align*}
& |u_n|^{q(x)} \rightharpoonup \nu = |u|^{q(x)} + \sum_{i\in I} \nu_i \delta_{x_i},\quad \text{weakly in the sense of measures,}\\
&|\nabla u_n|^{p(x)} \rightharpoonup \mu \ge |\nabla u|^{p(x)} + \sum_{i\in I} \mu_i \delta_{x_i},\quad \text{weakly in the sense of measures,}
\end{align*}
where $u\in W^{1,p(x)}_0(\Omega)$, $I$ is a finite set, $x_i\in \A$ and $\overline{S}_{x_i}^{ -1} \mu_i^{1/p(x_i)}\ge \nu_i^{1/p^*(x_i)}$

Hence, using Theorem \ref{desigualdad.primera},
\begin{align*}
1&=\lim_{n\to\infty} \int_{\Omega}\frac{|\nabla u_n|^{p(x)}}{\|\nabla u_n\|_{L^{p(x)}(\Omega)}^{p(x)}}\,dx\\
&\geq \int_{\Omega} |S(p(\cdot),q(\cdot),\Omega)^{-1} \nabla u|^{p(x)}\, dx + \sum_{i\in I} S(p(\cdot),q(\cdot),\Omega)^{-p(x_i)}\mu_{i}\\
&\geq  \int_{\Omega} |S(p(\cdot),q(\cdot),\Omega)^{-1} \nabla u|^{p(x)}\, dx + \sum_{i\in I}\overline{S}_{x_i} ^{-p(x_i)}\mu_{i}\\
&\geq \min\{(S(p(\cdot),q(\cdot),\Omega)^{-1} \|\nabla u\|_{L^p(x)(\Omega)})^{p^+}, (S(p(\cdot),q(\cdot),\Omega)^{-1} \|\nabla u\|_{L^p(x)(\Omega)})^{p^-}\} + \sum_{i\in I} \nu_i^{\frac{p(x_i)}{p^*(x_i)}}\\
&\geq \min\{ \|u\|_{L^{q(x)}(\Omega)}^{p^+}, \|u\|_{L^{q(x)}(\Omega)}^{p^-}\} + \sum_{i\in I} \nu_i^{\frac{p(x_i)}{p^*(x_i)}}
\end{align*}
where in the last inequality we have used the definition of $S$ \eqref{constante 2}.

Now, as $\|u_n\|_{L^{q(x)}(\Omega)}=1$ and $u_n\rightharpoonup u$ weakly in $L^{q(x)}(\Omega)$, it follows that $\|u\|_{L^{q(x)}(\Omega)}\le 1$, hence
$$
\min\{ \|u\|_{L^{q(x)}(\Omega)}^{p^+}, \|u\|_{L^{q(x)}(\Omega)}^{p^-}\}
 = \|u\|_{L^{q(x)}(\Omega)}^{p^+}\geq \rho_{q}(u)^{\frac{p^+}{q^-}}.
$$
So we find that
\begin{equation}\label{grad.1}
\rho_{q}(u)^{\frac{p^+}{q^-}} + \sum_{i\in I} \nu_i^{\frac{p(x_i)}{p^*(x_i)}} \leq 1.
\end{equation}

On the other hand, as $u_n$ is normalized, we get that
\begin{equation}\label{grad.2}
1 =  \rho_{q}(u) + \sum_{i\in I} \nu_i.
\end{equation}
Since $p^+< q^-$, by \eqref{grad.1} and \eqref{grad.2}, we can conclude that either $\rho_q(u)=1$ and the set $I$ is empty, or $u=0$ and the set $I$ contains a single point.

If the first case occurs, then $1 = \|u_n\|_{L^{q(x)}(\Omega)} = \rho_q(u_n) = \rho_q(u) = \|u\|_{L^{q(x)}(\Omega)}$ and, as $L^{q(x)}(\Omega)$ is a strictly convex Banach space, it follows that $u_n\to u$ strongly in $L^{q(x)}(\Omega)$.

If the second case occurs it easily follows that $\nu_0=1$ and $\mu_0=\overline{S}_{x_0}^{ p(x_0)}$.
\end{proof}

With the aid of this result, we are now ready to prove Theorem \ref{existencia}.

\begin{proof}[Proof of Theorem \ref{existencia}]
Let $\{u_n\}_{n\in\N}$ be a minimizing sequence for \eqref{constante 2}.

If $\{u_n\}_{n\in\N}$ has a strongly convergence subsequence in $L^{q(x)}(\Omega)$, then the result holds.

Assume that this is not the case. Then, by the previous Theorem, there exists $x_0\in\A$ such that
$|u_n|^{q(x)}\rightharpoonup \delta_{x_0}$ weakly in the sense of measures and $|\nabla u_n|^{p(x)}\rightharpoonup \overline{S}_{x_0}^{ p(x_0)}\delta_{x_0}$ weakly in the sense of measures

So, for $\varepsilon>0$, we have
$$
\int_\Omega\Big( \frac{|\nabla u_n|}{\overline{S}_{x_0}-\varepsilon}\Big)^{p(x)}\,dx\to\frac{\overline{S}_{x_0}^{ p(x_0)}}{\Big(\overline{S}_{x_0}-\varepsilon\Big)^{p(x_0)}}>1
$$
Then, there exists $n_0$ such that for all $n\geq n_0$, we know that:
$$
\|\nabla u_n\|_{L^{p(x)}(\Omega)}>\overline{S}_{x_0}-\varepsilon
$$
Taking limit, we obtain
$$
S(p(\cdot),q(\cdot),\Omega) \geq \overline{S}_{x_0}-\varepsilon
$$

As $\varepsilon>0$ is arbitrary, the result follows.
\end{proof}

\section{Continuity of the Sobolev constant with respect to $p$ and $q$}
In this section, we prove the continuity of the Sobolev constant $S(p(\cdot), q(\cdot), \Omega)$ with respect to $p$ and $q$ in the $L^\infty(\Omega)$ topology for monotone sequences.

We first prove an easy Lemma on the continuity of the Rayleigh quotient.
\begin{lema}\label{contQ}
Let $p_n\to p$ and $q_n\to q$ in $L^\infty(\Omega)$. Then, for every $v\in C^\infty_c(\Omega)$, $Q_{p_n,q_n,\Omega}(v)\to Q_{p,q,\Omega}(v)$.
\end{lema}
\begin{proof}
We only need to prove that
$$
\| \nabla v\|_{L^{p_n(x)}(\Omega)}\to \|\nabla v\|_{L^{p(x)}(\Omega)} \qquad \mbox{and}\qquad
\| v\|_{L^{q_n(x)}(\Omega)}\to \| v\|_{L^{q(x)}(\Omega)}.
$$
For that, we have
$$
\int_\Omega\Big(\frac{|v|}{\| v\|_{L^{q(x)}(\Omega)}+\delta}\Big)^{q_n(x)}\,dx\to\int_\Omega\Big(\frac{|v|}{\| v\|_{L^{q(x)}(\Omega)}+\delta}\Big)^{q(x)}\,dx<1,
$$
so, there exist $n_0$ such that $\forall n\geq n_0$,
$$
\int_\Omega\Big(\frac{|v|}{\| v\|_{L^{q(x)}(\Omega)}+\delta}\Big)^{q_n(x)}\,dx<1.
$$
Therefore $\| v\|_{L^{q_n(x)}(\Omega)}\leq\| v\|_{L^{q(x)}(\Omega)}+\delta$. Analogously, we obtain $\| v\|_{L^{q(x)}(\Omega)}-\delta\leq\| v\|_{L^{q_n(x)}(\Omega)}$.
In conclusion, for every $\delta>0$ we get
$$
\| v\|_{L^{q(x)}(\Omega)}-\delta\leq\liminf\| v\|_{L^{q_n(x)}(\Omega)}\leq\limsup\| v\|_{L^{q_n(x)}(\Omega)}\leq\| v\|_{L^{q(x)}(\Omega)}+\delta
$$
In a complete analogous fashion, we get
$$
\|\nabla v\|_{L^{p(x)}(\Omega)}-\delta\leq\liminf \|\nabla v\|_{L^{p_n(x)}(\Omega)}\leq \limsup\|\nabla v\|_{L^{p_n(x)}(\Omega)}\leq \|\nabla v\|_{L^{p(x)}(\Omega)}+\delta
$$
This finishes the proof.
\end{proof}

Now we prove the main result of the section.

\begin{teo}\label{contS}
Let $p_n\to p$ and $q_n\to q$ in $L^\infty(\Omega)$. Assume, moreover, that $p_n\ge p$ and that $q_n\le q$.
Then $S (p_n(\cdot),q_n(\cdot),\Omega)\to S(p(\cdot),q(\cdot),\Omega)$.
\end{teo}

\begin{proof}
Given $\delta>0$ we pick $u\in C^\infty_c(\Omega)$ such that $Q_{p,q,\Omega}(u)\le S(p(\cdot),q(\cdot),\Omega)+\delta$. Since, by Lemma \ref{contQ},
$\lim_{n\to \infty} Q_{p_n,q_n,\Omega}(u) = Q_{p,q,\Omega}(u)$, we obtain, using $u$ as a test-function to estimate $S(p_n(\cdot),q_n(\cdot),\Omega)$, that
\begin{align*}
\limsup_{n\to \infty} S(p_n(\cdot),q_n(\cdot),\Omega) &\le \limsup_{n\to \infty} Q_{p_n,q_n,\Omega}(u)\\
&  = Q_{p,q,\Omega}(u)\\
&\le S(p(\cdot),q(\cdot),\Omega)+\delta
\end{align*}
for any $\delta>0$. It follows that
$$ \limsup_{n\to \infty} S(p_n(\cdot),q_n(\cdot),\Omega) \le S(p(\cdot),q(\cdot),\Omega).$$

We now claim that there holds
$$ \liminf_{n\to \infty} S(p_n(\cdot),q_n(\cdot),\Omega) \ge S(p(\cdot),q(\cdot),\Omega).$$
The claim will follow if we prove that for any $u\in C^\infty_c(\Omega)$,
\begin{equation}\label{Equ1}
  \|\nabla u\|_{L^{p_n(x)}(\Omega)}\ge (1+o(1))\|\nabla u\|_{L^{p(x)}(\Omega)},
\end{equation}
and
\begin{equation}\label{Equ2}
  \|u\|_{L^{q_n(x)}(\Omega)}\le (1+o(1))\|u\|_{L^{q(x)}(\Omega)},
\end{equation}
where  $o(1)$ is uniform in $u$.
Since $p_n\ge p$ we can use H\"older inequality (Theorem \ref{Holder}), with $\frac{1}{p}=\frac{1}{p_n}+\frac{1}{s_n}$ to obtain
\begin{equation*}
\begin{split}
 \|\nabla u\|_{L^{p(x)}}
 & \le \Big( (p/p_n)^+ + (p/s_n)^+ \Big) \|\nabla u\|_{L^{p_n(x)}} \|1\|_{L^{s_n(x)}}  \\
 & \le (1+o(1)) \|\nabla u\|_{L^{p_n(x)}} \max\{ |\Omega|^{(1/s_n)^+}, |\Omega|^{(1/s_n)^-} \} \\
 & = (1+o(1)) \|\nabla u\|_{L^{p_n(x)}},
\end{split}
\end{equation*}
where the $o(1)$ are uniform in $u$. Equation (\ref{Equ1}) follows. We prove (\ref{Equ2}) in the same way considering
$t_n=\frac{q_n q}{q-q_n}$ and writing that
\begin{equation*}
\begin{split}
 \|v\|_{L^{q_n(x)}}
 & \le \Big( (q_n/q)^+ + (q_n/t_n)^+ \Big) \|v\|_{L^{q(x)}} \|1\|_{L^{t_n(x)}}  \\
 & = (1+o(1)) \|v\|_{L^{q(x)}}.
\end{split}
\end{equation*}
The proof is now complete.
\end{proof}


\section{Investigation on the validity of  $\bar S = \inf_{p^-_\A\le r\le p^+_\A} K_r^{-1}$}
In this section we investigate whether the equality
\begin{equation}\label{equality}
\bar S = \inf_{p^-_\A\le r\le p^+_\A} K_r^{-1}
\end{equation}
holds or not.

We show that, under certain assumptions on $p(x_0)$ and $q(x_0)$, $x_0\in\A$ the equality
\begin{equation}\label{equality2}
\bar S_{x_0} = K^{-1}_{p(x_0)}
\end{equation}
is valid.

As far as we know, it is an open problem to determine wether the equality holds true or not in general.

The aim of this section is to prove the following Theorem.
\begin{teo}\label{Conv} Assume that $p(\cdot)$ and $p^*(\cdot)/q(\cdot)$ have a strict local minimum at $x_0\in \A$. Then
$$
\lim_{\varepsilon \to 0} S(p(\cdot),q(\cdot),B_\varepsilon) = K_{p(x_0)}^{-1}.
$$
\end{teo}

This Theorem is a direct consequence of Theorem \ref{contS} and the following result:

\begin{prop}\label{Dilat} Assume $0\in \A$ and denote by $p=p(0)$, $B_\varepsilon = B_\varepsilon(0)$.

For any $u\in C^\infty_c(B_\varepsilon)$, there holds
$$
\|u\|_{L^{q(x)}(B_\varepsilon)} = \varepsilon^{N/p^*}(1+o(1)) \|u_\varepsilon\|_{L^{q_\varepsilon(x)}(B_1)},
$$
and
$$
\|\nabla u\|_{L^{p(x)}(B_\varepsilon)} = \varepsilon^{N/p^*}(1+o(1)) \|\nabla u_\varepsilon\|_{L^{p_\varepsilon(x)}(B_1)},
 $$
where $o(1)$ is uniform in $u$, $p_\varepsilon(x):=p(\varepsilon x)$, $q_\varepsilon(x):=q(\varepsilon x)$ and $u_\varepsilon(x):=u(\varepsilon x)$.
\end{prop}

Assuming Proposition \ref{Dilat} we can prove Theorem \ref{Conv}.
\begin{proof}[Proof of Theorem \ref{Conv}]
We have
$$ Q(p(\cdot),q(\cdot),B_\varepsilon)(u) = (1+o(1)) Q(p_\varepsilon(\cdot),q_\varepsilon(\cdot),B_1)(u_\varepsilon) $$
where the $o(1)$ is uniform in $u$, so that, noticing that the map $u\in C^\infty_c(B_\varepsilon)\mapsto u_\varepsilon\in C^\infty_c(B_1)$ is bijective,
\begin{equation*}
\begin{split}
  S(p(\cdot),q(\cdot),B_\varepsilon) & = (1+o(1)) S(p_\varepsilon(\cdot),q_\varepsilon(\cdot),B_1) = (1+o(1)) S(p(0),q(0),B_1) \\
                      & = (1+o(1)) S(p(0),p(0)^*,B_1)
\end{split}
\end{equation*}
which proves Theorem \ref{Conv}.
\end{proof}

It remains to prove Proposition \ref{Dilat}.

\begin{proof}[Proof of  Proposition \ref{Dilat}]

Given $u\in C^\infty_c(B_\varepsilon)$ we have
$$ \|u\|_{L^{q(x)}(B_\varepsilon)} = \inf\,\{\lambda>0\colon I_{q}^{\lambda, \varepsilon}(u)\le 1\}, $$
where
$$ I_{q}^{\lambda,\varepsilon}(u) := \int_{B_\varepsilon} \Big|\frac{u(x)}{\lambda}\Big|^{q(x)}\,dx
= \int_{B_1}\Big|\frac{u_\varepsilon(x)}{\lambda \varepsilon^{-\frac{N}{q_\varepsilon(x)}}}\Big|^{q_\varepsilon(x)}\,dx.
$$
Writing that
$$ \varepsilon^{-\frac{N}{q_\varepsilon(x)}} 
= \exp\{-N\ln\varepsilon (q(0)+O(\varepsilon))^{-1}\} =  \varepsilon^{-N/p^*}(1+o(1)),  $$
where the $O(\varepsilon)$ and the $o(1)$ are uniform in $x$ and $u$,
we obtain
\begin{equation}
\begin{split}
  \|u\|_{L^{q(x)}(B_\varepsilon)} & = \inf\,\{\lambda>0\colon I_{q}^{\lambda, \varepsilon}(u)\le 1\} \\
  & = \varepsilon^{N/p^*}(1+o(1)) \inf\,\{\tilde \lambda>0\colon I_{q_\varepsilon}^{\tilde \lambda, 1}(u_\varepsilon) \le 1\}
\end{split}
\end{equation}
from which we deduce the result. The proof of the result for the gradient term is similar: we have
$$ \| \nabla u\|_{L^{p(x)}(B_\varepsilon)} = \inf\,\{\lambda>0\colon I_p^{\lambda, \varepsilon}(\nabla u)\le 1\}, $$
and
$$ I_p^{\lambda, \varepsilon}(\nabla u) = \int_{B_\varepsilon} \Big|\frac{\nabla u(x)}{\lambda}\Big|^{p(x)}\,dx
= \int_{B_1}\Big|\frac{\nabla u_\varepsilon(x)}{\lambda \varepsilon^{1-\frac{N}{p_\varepsilon(x)}}}\Big|^{p_\varepsilon(x)}\,dx,
$$
and we can end the proof as before.
\end{proof}

\section{On the strict inequality $S(p(\cdot), q(\cdot),\Omega) < \bar S$}

In this section we provide with an example of a domain $\Omega$ and exponents $p,q$ where the condition $S(p(\cdot),q(\cdot),\Omega)<\overline{S}$ is satisfied.

The condition is the existence of a large ball where the exponent $q$ is subcritical. Up to our knowledge it is not known if $S(p(\cdot),q(\cdot),\Omega)<\overline{S}$ can hold when $q\equiv p^*$ on $\Omega$.

This example somewhat relates to the one analyzed in \cite{MOSS}.
More precisely, we can show

\begin{teo}\label{ejemplo}
Assume that $B_R\subset \Omega\setminus \A$ where $B_R$ is a ball of radius $R$. Moreover, assume that $q^+_{B_R} < (p^*)^-_{B_R}$.

Then, If $R$ is large enough, we have that $S(p(\cdot),q(\cdot),\Omega)<\overline{S}$.
\end{teo}

\begin{proof}
Assume that $\Omega$ contains a subcritical ball $B_R$. 
Take $u\in C^\infty_c(B_1)$ such that $|u|,|\nabla u|\le 1$, and consider $u_R(x)=u(x/R)$. 
We take $R$ big enough to have  
$$ 
R^{N-p^+} \int_{B_1} |\nabla u|^{p^+} \, dx > 1,\quad R^N\int_{B_1} |u|^{q^+} > 1 
$$
and 
$$ 
\frac{\|\nabla u\|_{L^{p^-_{B_R}}(B_R)}}{\|u\|_{L^{q^+_{B_R}}(B_R)}} R^{N(1/(p^-_{B_R})^*-1/q^+_{B_R})} < S.  
$$
Then we claim that 
$$ 
\frac{\|\nabla u_R\|_{L^{p(x)}(B_R)}}{\|u_R\|_{L^{q(x)}(B_R)}} < S.
$$ 
We first note that 
$$ 
\int_{B_R} |\nabla u_R|^{p(x)}\,dx = \int_{B_1} R^{N-p(Rx)}|\nabla u|^{p(Rx)}(x)\,dx
 \ge R^{N-p^+} \int_{B_1} |\nabla u|^{p^+} \, dx   > 1 
$$ 
so that, by Proposition \ref{norma.y.rho}, 
$$ 
\|\nabla u_R\|_{L^{p(x)}(B_R)} \le \left( \int_{B_R} |\nabla u_R|^{p(x)}\,dx \right)^{1/p^-_{B_R}}
   \le R^\frac{N-p^-_{B_R}}{p^-_{B_R}} \left( \int_{B_1} |\nabla u|^{p^-_{B_R}} \,dx \right)^{1/p^-_{B_R}}.
$$ 
In the same way 
$$ 
\int_{B_R} |u_R|^{q(x)}\,dx = R^N\int_{B_1} |u|^{q(Rx)}\,dx \ge R^N\int_{B_1} |u|^{q^+} > 1 
$$ 
so that 
$$ 
\|u_R\|_{L^{q(x)}(B_R)} \ge \left( \int_{B_R} |u_R|^{q(x)}\,dx \right)^{1/q^+_{B_R}} 
                      \ge R^{N/q^+_{B_R}} \|u\|_{L^{q^+_R}}, 
$$ 
from which we deduce our claim. This finishes the proof.
\end{proof}

\section*{Acknowledgements}
This work was partially supported by Universidad de Buenos Aires under grant X078 and by CONICET (Argentina) PIP 5478/1438.

\def\ocirc#1{\ifmmode\setbox0=\hbox{$#1$}\dimen0=\ht0 \advance\dimen0
  by1pt\rlap{\hbox to\wd0{\hss\raise\dimen0
  \hbox{\hskip.2em$\scriptscriptstyle\circ$}\hss}}#1\else {\accent"17 #1}\fi}
  \def\ocirc#1{\ifmmode\setbox0=\hbox{$#1$}\dimen0=\ht0 \advance\dimen0
  by1pt\rlap{\hbox to\wd0{\hss\raise\dimen0
  \hbox{\hskip.2em$\scriptscriptstyle\circ$}\hss}}#1\else {\accent"17 #1}\fi}
\providecommand{\bysame}{\leavevmode\hbox to3em{\hrulefill}\thinspace}
\providecommand{\MR}{\relax\ifhmode\unskip\space\fi MR }
\providecommand{\MRhref}[2]{%
  \href{http://www.ams.org/mathscinet-getitem?mr=#1}{#2}
}
\providecommand{\href}[2]{#2}

\end{document}